\RequirePackage{fix-cm}
\documentclass[smallextended]{svjour3}       
\smartqed  

\usepackage{mathtools}
\usepackage[normalem]{ulem} 

\usepackage{amssymb}
\usepackage{mathrsfs}

\newcommand{\cut}[1]{{}}

\newcommand{\va}{{\mathbf{a}}}
\newcommand{\vb}{{\mathbf{b}}}

\newcommand{\vp}{{\mathbf{p}}}
\newcommand{\vq}{{\mathbf{q}}}

\newcommand{\vs}{{\mathbf{s}}}
\newcommand{\vt}{{\mathbf{t}}}

\newcommand{\vx}{{\mathbf{x}}}
\newcommand{\vy}{{\mathbf{y}}}
\newcommand{\vz}{{\mathbf{z}}}

\newcommand{\vA}{{\mathbf{A}}}

\newcommand{\vD}{{\mathbf{D}}}

\newcommand{\vI}{{\mathbf{I}}}

\newcommand{\vK}{{\mathbf{K}}}

\newcommand{\vM}{{\mathbf{M}}}

\newcommand{\vP}{{\mathbf{P}}}

\newcommand{\vR}{{\mathbf{R}}}

\newcommand{\vT}{{\mathbf{T}}}

\newcommand{\vW}{{\mathbf{W}}}

\newcommand{\cH}{{\mathcal{H}}}

\newcommand{\cS}{{\mathcal{S}}}

\newcommand{\cX}{{\mathcal{X}}}
\newcommand{\cY}{{\mathcal{Y}}}

\newcommand{\vzero}{\mathbf{0}}
\newcommand{\vone}{{\mathbf{1}}}

\newcommand{\dom}{{\mathrm{dom}}} 

\newcommand{\prox}{\mathbf{prox}}

\newcommand{\Null}{\mathbf{Null}}
\newcommand{\Span}{\mathbf{span}}

\makeatletter
\let\@@span\span
\def\sp@n{\@@span\omit\advance\@multicnt\m@ne}
\makeatother

\DeclareMathOperator*{\argmin}{arg\,min}

\DeclareMathOperator*{\Min}{minimize}

\newcommand{\bc}{\begin{center}}
\newcommand{\ec}{\end{center}}

\newcommand{\bdm}{\begin{displaymath}}
\newcommand{\edm}{\end{displaymath}}

\newcommand{\beq}{\begin{equation}}
\newcommand{\eeq}{\end{equation}}

\newcommand{\bfl}{\begin{flushleft}}
\newcommand{\efl}{\end{flushleft}}

\newcommand{\bt}{\begin{tabbing}}
\newcommand{\et}{\end{tabbing}}

\newcommand{\beqn}{\begin{align}}
\newcommand{\eeqn}{\end{align}}

\newcommand{\beqs}{\begin{align*}} 
\newcommand{\eeqs}{\end{align*}}  

\newtheorem{assumption}{Assumption}

\usepackage{graphicx}
\usepackage[colorlinks=true,breaklinks=true,bookmarks=true,urlcolor=blue,
citecolor=blue,linkcolor=blue,bookmarksopen=false,draft=false]{hyperref}
\def\email#1{\href{mailto:#1}{#1}}
\begin{document}
	
\title{New convergence analysis of a primal-dual algorithm with large stepsizes \thanks{This work was supported in part by the National Science Foundation (NSF) grants DMS-1621798 and DMS-2012439, the Natural Science Foundation of China grant 62001167.
}}
\titlerunning{a primal-dual algorithm with large stepsizes}        


\author{Zhi Li         \and
	Ming Yan 
}


\institute{Z. Li \at
	Shanghai Key Laboratory of Multidimensional Information Processing, Department of Computer Science and Technology, East China Normal University, Shanghai 200062, China,\\ \email{lizhiupc@gmail.com}    
	\and
	M. Yan \at
	Department of Computational Mathematics, Science and Engineering and Department of Mathematics,
	Michigan State University, East Lansing, MI 48824, USA, \email{myan@msu.edu}
}

\date{Received: date / Accepted: date}
\maketitle

\begin{abstract}
	We consider a primal-dual algorithm for minimizing $f(\vx)+h\square l(\vA\vx)$ with Fr\'echet differentiable $f$ and $l^*$. This primal-dual algorithm has two names in literature: Primal-Dual Fixed-Point algorithm based on the Proximity Operator (PDFP$^2$O) and Proximal Alternating Predictor-Corrector (PAPC). In this paper, we prove its convergence under a weaker condition on the stepsizes than existing ones. With additional assumptions, we show its linear convergence. In addition, we show that this condition (the upper bound of the stepsize) is tight and can not be weakened. This result also recovers a recently proposed positive-indefinite linearized augmented Lagrangian method. In addition, we apply this result to a decentralized consensus algorithm PG-EXTRA and derive the weakest convergence condition.
		
\end{abstract}

\section{Introduction}\label{sec:intro}
Minimizing the sum of two functions has applications in various areas including image processing, machine learning, and decentralized consensus optimization~\cite{chambolle2011first,chambolle2016introduction,jaggi2014communication,shi2015proximal}. In this paper, we aim to minimize the sum of two functions in the following form:
\begin{align}\label{for:main_problem}
\Min_{\vx\in\cX}~f(\vx) + h\square l(\vA\vx),
\end{align}
where $\cX$ and $\cS$ are two real Hilbert spaces; $f(\vx):\cX\mapsto(-\infty,+\infty]$, $h(\vs):\cS\mapsto(-\infty,+\infty]$, and $l(\vs):\cS\mapsto(-\infty,+\infty]$ are proper lower semi-continuous (lsc) convex functions; $h\square l$ is the infimal convolution of $h$ and $l$ that is defined as $h\square l(\vs)=\inf_{\vt\in\cS}~h(\vt)+l(\vs-\vt)$; the linear operator $\vA:\cX\mapsto \cS$ is bounded. In addition, we assume that $f(\vx)$ is Fr\'echet differentiable with a Lipschitz continuous gradient, $l$ is strongly convex in $\dom(l)$\footnote{It means that $l^*(\vs)$ (the Legendre-Fenchel conjugate of $l(\vs)$) is Fr\'echet differentiable with a Lipschitz continuous gradient.}, and the proximal operator of $h$, which is defined as
\begin{align*}
\prox_{\lambda h}(\vt)=(\vI+\lambda\partial h)^{-1}(\vt):=\argmin_{\vs\in\cS}~h(\vs)+{1\over {2\lambda}}\|\vs-\vt\|^2,
\end{align*}
has a closed-form solution or can be easily computed. Here $\partial h$ is the subdifferential of the convex function $h$.

Many existing papers considered a special case of~\eqref{for:main_problem} with $l(\vs)$ being the indicator function $\iota_{\{\vzero\}}(\vs)$ that returns 0 if $\vs=\vzero$ and $+\infty$ otherwise.  In this special case, the infimal convolution $h\square l$ degenerates to $h$, and the problem~\eqref{for:main_problem} becomes
\begin{align}\label{for:main_problem_noInf}
\Min_{\vx\in\cX}~f(\vx) + h(\vA\vx).
\end{align}
The corresponding saddle-point problem is  
\begin{align}\label{for:he_saddle}
\min_{\vx\in\cX}\max_{\vs\in\cS}~f(\vx)+\langle \vA\vx,\vs\rangle-h^*(\vs).
\end{align}
If a saddle point $(\vx^\star,\vs^\star)$ exists for~\eqref{for:he_saddle}, then $\vx^\star$ is an optimal solution for~\eqref{for:main_problem_noInf}. 

In order to solve~\eqref{for:main_problem_noInf} (or~\eqref{for:he_saddle}), a primal-dual algorithm was proposed in different fields under different names~\cite{chen2013primal,drori2015simple,loris2011generalization}. Loris and Verhoeven~\cite{loris2011generalization} focused on a particular smooth function $f(\vx)={1\over 2}\|\vK\vx-\vy\|^2$, where $\vK$ is a linear operator and $\vy$ is given. Chen, Huang, and Zhang~\cite{chen2013primal} considered the general problem~\eqref{for:main_problem_noInf} and proposed a Primal-Dual Fixed-Point algorithm based on the Proximity Operator (PDFP$^2$O). Then the same algorithm was rediscovered under the name Proximal Alternating Predictor-Corrector (PAPC) in~\cite{drori2015simple} to solve~\eqref{for:main_problem_noInf} and its extension to a finite sum of composite functions when $h$ is separable. One iteration of the algorithm is
\begin{subequations}\label{for:PD3O_nog_simp}
	\begin{align}
	\vs^{k+1}     & =\left(\vI+{\sigma}\partial h^*\right)^{-1} \left((\vI-\tau\sigma\vA\vA^{\top}) \vs^k + \sigma\vA\left(\vx^k-{\tau}\nabla f(\vx^k)\right)\right), \\
	\vx^{k+1}     & =\vx^k-{\tau}\nabla f(\vx^k) -{\tau}\vA^\top\vs^{k+1}.
	\end{align}
\end{subequations}
Here $\tau$ and $\sigma$ are the primal and dual stepsizes, respectively, and the convergence of this algorithm is shown when $\tau\sigma\|\vA\vA^\top\|\leq 1$ and $2\tau/L< 1$~\cite{chen2013primal}. Here $L$ is the Lipschitz constant of $\nabla f(\vx)$ and $\|\vA\vA^\top\|$ is the operator norm of $\vA\vA^\top$. When $\vA$ is a matrix, $\|\vA\vA^\top\|$ is the largest eigenvalue of $\vA\vA^\top$.

There are many other algorithms for solving~\eqref{for:main_problem_noInf} and its extensions. For example, Condat-Vu~\cite{chambolle2016ergodic,condat2013primal,vu2013splitting} solves a more general problem than~\eqref{for:main_problem_noInf} with an additional non-differential function. However, the corresponding parameters $\tau$ and $\sigma$ have to satisfy $\tau\sigma\|\vA\vA^\top\|+2\tau/L \leq 1$~\cite{2017arXiv170206234K}. When $f(\vx)= 0$, Condat-Vu reduces to Chambolle-Pock~\cite{chambolle2011first}. There are several other primal-dual algorithms for minimizing the sum of three functions, one of which differentiable~\cite{chen2016primal,yan2017three,bot2013douglas,boct2015convergence,combettes2012primal,latafat2017asymmetric,davis2017three}. Interested readers are referred to~\cite{komodakis2015playing,yan2017three} for the comparison of different primal-dual algorithms for minimizing the sum of three functions. All the algorithms mentioned above solve bilinear saddle-point problems in the form of~\eqref{for:he_saddle} or its variants. Recently, many algorithms have been developed to solve more general saddle-point problems with non-bilinear terms~\cite{hamedani2018iteration,hamedani2018primal,hien2017inexact,xu2017first,xu2018primal}. A review for primal-dual algorithms is beyond the scope of this paper, and we focus on the specific primal-dual algorithm PAPC here. 

When there is only one function $f(\vx)$, i.e., $h(\vs)=0$, we let $\vA=\vzero$, and the primal-dual algorithm reduces to the gradient descent with stepsize $\tau$. Therefore, the condition $\tau<2/L$ can not be relaxed. The remaining question is {\it can the condition $\tau\sigma\leq 1/\|\vA\vA^\top\|$ be relaxed?} In~\cite[Section 5.1]{chen2013primal}, the authors numerically showed that a larger stepsize (e.g., $\tau\sigma=4/(3\|\vA\vA^\top\|)$) gives a better performance than stepsizes satisfying the condition $\tau\sigma\leq 1/\|\vA\vA^\top\|$. The convergence for $\tau\sigma<4/(3\|\vA\vA^\top\|)$ was an open problem, and this work resolves it.

For linearized Augmented Lagrangian Method (ALM)~\cite{yang2013linearized}--a special case of the primal-dual algorithm~\eqref{for:PD3O_nog_simp}--the condition $\tau\sigma\leq 1/\|\vA\vA^\top\|$ is relaxed in~\cite{he2016positive}. Consider the constrained optimization problem
\begin{subequations}\label{eq:he_dual} 
	\begin{align*} 
	\Min\limits_{\vs}~&h^*(\vs),	\\
	\textnormal{subject to}~&  -\vA^\top\vs=\vb.
	\end{align*}
\end{subequations}
Its dual problem is 
\begin{align*}
\Min_\vx~\vb^\top\vx+h(\vA\vx),
\end{align*}
which is the problem~\eqref{for:main_problem_noInf} with $f(\vx)=\vb^\top\vx$. The linearized ALM is 
\begin{subequations}\label{for:he_linearizedALM}
	\begin{align}
	\vs^{k+1}     & = \argmin_{\vs}~h^*(\vs) + {\beta\over {2}}\left\|\vs-\vs^k-{1\over {\beta}}\vA(\vx^k-\tau(\vA^\top\vs^k+\vb))\right\|^2, \label{for:he_linearizedALM_a}\\
	\vx^{k+1}     & =\vx^k-\tau(\vA^\top\vs^{k+1}+\vb).
	\end{align}
\end{subequations}
It is exactly the primal-dual algorithm~\eqref{for:PD3O_nog_simp} with $\beta=1/\sigma$. Note that the step in~\eqref{for:he_linearizedALM_a} can be rewritten as
\begin{align*}
\argmin_{\vs}~h^*(\vs) -\langle \vx^k,\vA^\top\vs+\vb\rangle+{\tau\over2}\|\vA^\top\vs+\vb\|_2^2+{1\over2\sigma}\|\vs-\vs^k\|_{\vI-\tau\sigma\vA\vA^\top}^2.
\end{align*}
In~\cite{yang2013linearized}, positive-definiteness of $\vI-\tau\sigma\vA\vA^\top$ is required for showing the convergence. Then the authors in~\cite{he2016positive} relaxed the condition and showed that $(4/3)\vI-\tau\sigma\vA\vA^\top$ being positive definite is the necessary and sufficient condition for the convergence of linearized ALM. That is, this relaxed condition is sufficient for the convergence of linearized ALM, and if the condition is not satisfied, there exists a function $h^*(\vs)$, a linear operator $\vA$, and an initialization such that the algorithm does not converge. This result motivates us to show the convergence of~\eqref{for:PD3O_nog_simp} under a weaker condition. In this paper, we provide the necessary and sufficient condition on $\tau\sigma$ for the convergence of algorithm~\eqref{for:PD3O_nog_simp}. This extension from~\cite{he2016positive} is nontrivial because the function $f(\vx)$ from linearized ALM is linear, i.e., $f(\vx)=\vb^\top\vx$,  and the Lipschitz constant of $\nabla f$ is 0.

Furthermore, we consider the more general problem~\eqref{for:main_problem} with infimal convolution, which was not considered in~\cite{chen2013primal,drori2015simple}, because it provides a tight upper bound for the stepsize of Proximal Gradient EXact firsT-ordeR Algorithm (PG-EXTRA) in decentralized consensus optimization. More details are in Section~\ref{sec:compare}.

In this paper, we relax the parameters for the primal-dual algorithm~\eqref{for:PD3O_nog_simp} and provide a tight bound for the primal and dual stepsizes. This result recovers one special case of the positive-indefinite ALM in~\cite{he2016positive}. Instead of using positive semidefinite operators for primal-dual variables in standard analysis, we allow the operator to be indefinite, see the operator in~\eqref{for:FB1}. Note that the analysis in this paper with indefinite operators is nontrivial because the standard techniques can not be applied. In addition, the linear convergence result is better than existing ones. Finally we apply this result to a decentralized consensus algorithm and obtain its weakest convergence condition.

The rest of this paper is organized as follows. In Section \ref{sec:convergence}, we present the algorithm to solve~\eqref{for:main_problem}. We show its convergence for the general case in Section~\ref{sec:convergence_gen} and linear convergence rates under additional assumptions in Section~\ref{sec:linear_conv}. In Section~\ref{sec:optimal_bdd}, we provide one example to show that the upper bound for its stepsize is tight. The application to a decentralized consensus algorithm is provided in Section~\ref{sec:compare}. Then we end this paper with a short conclusion.

\section{New convergence results with weaker conditions} \label{sec:convergence}
\subsection{A primal-dual algorithm}
In this paper, we extend an existing primal-dual algorithm~\eqref{for:PD3O_nog_simp} to solve~\eqref{for:main_problem} with an infimal convolution and show its convergence results with weaker conditions. Firstly, we explain this algorithm via operator splitting, which is different from those in the literature. Instead of considering problem~\eqref{for:main_problem}, we consider the corresponding saddle-point problem
\begin{align}\label{pro:saddle_point}
\min_\vx\max_\vs f(\vx) +\langle \vA\vx,\vs\rangle -h^*(\vs)-l^*(\vs),
\end{align}
whose optimality condition for a saddle point $(\vx^\star,\vs^\star)$ is
\begin{align*}
\begin{bmatrix} \vzero\\\vzero\end{bmatrix} \in \begin{bmatrix}
0 & \vA^\top \\-\vA & \partial h^*
\end{bmatrix}
\begin{bmatrix} \vx^\star\\\vs^\star\end{bmatrix} + \begin{bmatrix} \nabla f(\vx^\star)\\\nabla l^*(\vs^\star)\end{bmatrix}.
\end{align*}
We apply the following forward-backward operator splitting with self-adjoint positive definite operators $\vP$ and $\vD-\tau\sigma\vA\vP^{-1}\vA^\top$ defined on $\cX$ and $\cS$, respectively:
\begin{align}
&\begin{bmatrix}
\vP & \vzero \\\vzero & \vD-\tau\sigma\vA\vP^{-1}\vA^\top
\end{bmatrix}\begin{bmatrix} \vx^{k}\\\vs^{k}\end{bmatrix} -\begin{bmatrix} \tau\nabla f(\vx^k)\\\sigma\nabla l^*(\vs^k)\end{bmatrix} \nonumber\\
\in & \begin{bmatrix}
\vP & \vzero \\\vzero & \vD-\tau\sigma\vA\vP^{-1}\vA^\top 
\end{bmatrix}\begin{bmatrix} \vx^{k+1}\\\vs^{k+1}\end{bmatrix}+ \begin{bmatrix}
0 & \tau\vA^\top \\-\sigma\vA & \sigma\partial h^*
\end{bmatrix}
\begin{bmatrix} \vx^{k+1}\\\vs^{k+1}\end{bmatrix}. \label{for:FB1}
\end{align}
Here $\tau$ and $\sigma$ are two positive parameters. When $\vP$ and $\vD$ are the identity operators in $\cX$ and $\cY$, respectively, $\tau$ and $\sigma$ are the primal and dual stepsizes, respectively. Different operators $\vP$ and $\vD$ may be chosen in different scenarios. For example, we can choose $\vP$ (or $\vD$) to be a diagonal matrix such that the stepsize is different for different coordinates of $\vx$ (or $\vs$) when $\cX$ (or $\cS$) is finite dimensional. Define $\vM =\frac{\tau}{\sigma}(\vD-\tau\sigma\vA\vP^{-1}\vA^{\top})$. Then, we apply the Gaussian elimination and obtain 
\begin{align*}
\begin{bmatrix}
\vP & \vzero \\\sigma\vA & {\sigma\over\tau}\vM 
\end{bmatrix}\begin{bmatrix} \vx^{k}\\\vs^{k}\end{bmatrix} -\begin{bmatrix} \tau\nabla f(\vx^k)\\\sigma\tau\vA\vP^{-1}\nabla f(\vx^k)+\sigma\nabla l^*(\vs^k)\end{bmatrix} 
\in  \begin{bmatrix}
\vP & \tau\vA^\top \\\vzero & \vD +\sigma\partial h^*
\end{bmatrix}\begin{bmatrix} \vx^{k+1}\\\vs^{k+1}\end{bmatrix}.\end{align*}
Given $(\vx^k,\vs^k)$, one iteration of the primal-dual algorithm is 
\begin{subequations}\label{for:PD3O_nog}
	\begin{align}
	\vs^{k+1}     & =\left(\vD+\sigma\partial h^*\right)^{-1} \left(\frac{\sigma}{\tau}\vM \vs^k + \sigma\vA\left(\vx^k-{\tau}\vP^{-1}\nabla f(\vx^k)\right)-\sigma\nabla l^*(\vs^k)\right), \label{for:PD3O_nog_iteration_b}\\
	\vx^{k+1}    	& =\vx^k-{\tau}\vP^{-1}\nabla f(\vx^k) -{\tau}\vP^{-1}\vA^\top\vs^{k+1}.\label{for:PD3O_nog_iteration_c}
	\end{align}
\end{subequations}
From this analysis, we can easily see that a point $(\vx^\star,\vs^\star)$ is a saddle point of~\eqref{pro:saddle_point} if and only if it is a fixed point of~\eqref{for:PD3O_nog}. Therefore, we only need to show the convergence to a fixed point of~\eqref{for:PD3O_nog}. Note that we could store $\vA^\top\vs$ in the implementation, and the iteration is equivalent to
\begin{subequations}
	\begin{align*}
	\vs^{k+1}     & =\left(\vD+\sigma\partial h^*\right)^{-1} \left(\vD\vs^k  + \sigma\vA\left(\vx^k-{\tau}\vP^{-1}(\nabla f(\vx^k)+\vA^\top\vs^k)\right)-\sigma\nabla l^*(\vs^k)\right),\\
	\vx^{k+1}    	& =\vx^k-{\tau}\vP^{-1}\nabla f(\vx^k) -{\tau}\vP^{-1}\vA^\top\vs^{k+1}.
	\end{align*}
\end{subequations}
Therefore, only one application of $\vA$ and one application of $\vA^\top$ are needed in each iteration.

Let $\vI$ be the identity operator defined on a Hilbert space. For simplicity, we do not specify the space on which it is defined when it is clear from the context. When $l$ is the indicator of a singleton\footnote{It means that $\nabla l^*(\vs)\equiv 0$.}, $\vP=\vI$, and $\vD=\vI$, the iteration of~\eqref{for:PD3O_nog} reduces to~\eqref{for:PD3O_nog_simp}, the existing primal-dual algorithm proposed in~\cite{chen2013primal,drori2015simple,loris2011generalization}. Its convergence is shown if $\vI-\tau\sigma\vA\vA^{\top}$ is positive semidefinite and $\tau< 2/L$ with $L$ being the Lipschitz constant of $\nabla f$.

If the operators $\vP$ and $\vD-\tau\sigma\vA\vP^{-1}\vA^\top$ are positive definite, the convergence of~\eqref{for:PD3O_nog} with an additional condition for $\tau$ can be shown easily from nonexpansive operators with metric~\cite{yan2017three,bauschke2011convex,ryu2016primer}. To the best of our knowledge, this paper is the first one to show the convergence of a primal-dual algorithm when $\vD-\tau\sigma\vA\vP^{-1}\vA^\top$ is not positive definite, and the analysis is different from positive definite cases.

\subsection{Assumptions for new analysis}
An extension of this existing primal-dual algorithm~\eqref{for:PD3O_nog_simp} to~\eqref{for:PD3O_nog} is derived to solve the problem~\eqref{for:main_problem} with an infimal convolution. In addition, we show the convergence of~\eqref{for:PD3O_nog_simp} with a larger $\tau\sigma$. Specifically, we can choose $\tau\sigma$ such that $(4/3)\vD-\tau\sigma\vA\vP^{-1}\vA^{\top}$ is positive semidefinite, i.e., the upper bound for $\tau\sigma$ is increased by $1/3$.  It means that we can choose a larger stepsize $\sigma$ when the primal stepsize $\tau$ is fixed.

For convenience, we introduce two operators as
\begin{align*}
\vM_1 := {{\tau\over\sigma}(\vD-\theta\tau\sigma\vA\vP^{-1}\vA^\top)},\qquad \vM_2 := {{\tau^2}(1-\theta)\vA\vP^{-1}\vA^\top}.
\end{align*}
Here $\theta\in(3/4,1]$ is chosen such that $\vM_1$ is positive definite and $\vM_2$ is positive semidefinite. We can find such $\theta\in(3/4,1]$ whenever $(4/3)\vD-\tau\sigma\vA\vP^{-1}\vA^{\top}$ is positive semidefinite. We would like to emphasize here that $\theta>3/4$ is crucial in the proof of the convergence because we need $4\theta-3$ to be positive. On the other side, $\theta\leq 1$ is required for $\vM_2$ being positive semidefinite. With these two operators, we have $\vM=\vM_1-\vM_2$. In addition, we define a positive definite operator as follows
\begin{align*}
\widetilde\vM  := \vM_1+\vM_2.
\end{align*}
Given a self-adjoint operator $\overline\vM$, we let $\langle \vs,\vt\rangle_{\overline{\vM}}:=\langle \vs, \overline\vM\vt\rangle$ and $\|\vs\|_{\overline{\vM}}^2=\langle \vs, \overline\vM\vs\rangle$. Note that $\|\vs\|_{\overline{\vM}}^2$ can be negative if $\overline{\vM}$ is not positive semidefinite. When $\overline\vM$ is positive definite, we further define the induced norm as $\|\vs\|_{\overline\vM}=\sqrt{\langle\vs,\vs\rangle_{\overline\vM}}$. Let $\lambda_{\min}(\overline\vM)$ be the smallest eigenvalue of $\overline\vM$.  For $(\vx,\vs)\in\cX\times\cS$, we define $\|(\vx,\vs)\|_{\vP,{\overline{\vM}}}^2=\|\vx\|_\vP^2+\|\vs\|^2_{\overline{\vM}}$. 

\begin{assumption} \label{asp:1} Functions $f$, $h$, and $l$ are proper lsc convex. In addition, $f$ is Frechet differentiable and $l$ is strictly convex (i.e., $l^*$ is Frechet differentiable). Operators $\vP$ and $\vM_1$ are positive definite. The iteration~\eqref{for:PD3O_nog} has at least one fixed point. Let $(\vx^\star,\vs^\star)$ be any fixed point of~\eqref{for:PD3O_nog}. For any $\vx\in\cX$ and $\vs\in\cS$, we have
	\begin{align}
	\langle \vx-\vx^\star,\nabla f(\vx)-\nabla f(\vx^\star)\rangle\geq &\beta\|\nabla f(\vx)-\nabla f(\vx^\star)\|_{\vP^{-1}}^2,  \label{for:cocoer_f}\\
	\langle\vs-\vs^\star,\nabla l^*(\vs)-\nabla l^*(\vs^\star)\rangle\geq &\beta\|\nabla l^*(\vs)-\nabla l^*(\vs^\star)\|_{\vM_1^{-1}}^2,	 \label{for:cocoer_l}
	\end{align}
	for some $\beta>0$. 
\end{assumption}

\begin{lemma}
	When $f$ and $l^*$ have Lipschitz continuous gradients with parameters $L_f$ and $L_{l^*}$, respectively, we can choose 
	$$\beta = \min\left(\lambda_{\min}(\vP)L_f^{-1},{\tau\over\sigma}\lambda_{\min}(\vD-\theta\tau\sigma\vA\vP^{-1}\vA^\top)L_{l^*}^{-1}\right) $$ 
	such that Assumption~\ref{asp:1} is satisfied.
	When $\vD$ and $\vP$ are identity matrices, we can simplify it as 
	$$\beta = \min\left(L_f^{-1},{\tau\over\sigma}(1-\theta\tau\sigma\lambda_{\max}(\vA\vA^\top))L_{l^*}^{-1}\right).$$ 
\end{lemma}
The proof for this lemma is simple and omitted.

\begin{remark}\label{remark1}
We choose norms that are different from standard norms for simplicity. They come from the operators $\vP$ and $\vM$ in~\eqref{for:FB1}. 
\begin{itemize}
		\item The condition~\eqref{for:cocoer_f} usually comes from the cocoerciveness of $\nabla f$. It is satisfied with $\beta={\min_{\vx:\|\vx\|=1}\|\vP\vx\|\over L_f}$ if $f(\vx)$ has a Lipschitz continuous gradient with constant $L_f$~\cite[Theorem 18.15]{bauschke2011convex}. One example of $\vP$ is the diagonal matrix when $f$ is separable and the Lipschtiz continuous constants are different for different blocks. By choosing a diagonal matrix $\vP$ we can have a fast algorithm. For example, in~\cite{li2017decentralized}, we let different agent choose different stepsizes to improve the convergence speed.
		\item Note that the condition~\eqref{for:cocoer_l} depends on $\theta$, which does not exist in the algorithm. We choose to have the same $\beta$ in~\eqref{for:cocoer_f} and~\eqref{for:cocoer_l} for simplicity. From the definition of $\vM_1$, we can see that the condition~\eqref{for:cocoer_l} depends on function $l^*$, $\vP$, $\vD$, $\vA$, $\beta$,  $\theta$, $\tau$, and $\sigma$. But it is not as complicated as it looks like. Let's assume that $\vD=\vI$ and $\vP=\vI$, $f$ and $\l^*$ have Lipschitz continuous gradients with $L_f$ and $L_{l^*}$, respectively. The condition~\eqref{for:cocoer_l} requires 
		$$\beta\leq \lambda_{\min}(\vM_1)/L_{l^*}=\tau (1-\theta\tau\sigma\|\vA\vA^\top\|)/(\sigma L_{l^*}).$$
		Therefore, we can also choose a small $\theta\in(3/4,1]$ to make it valid if a larger $\beta$ works. By making $\theta$ small, we can have a large dual stepsize $\sigma$ for a given primal stepsize $\tau$. In fact, we do not need to know $\beta$ explicitly to determine both stepsizes. When we consider both conditions (\eqref{for:cocoer_f} and~\eqref{for:cocoer_l}) and the condition $\tau<2\beta$ in Theorem~\ref{thm:main}, we have 
		\begin{align}\tau L_f< 2,~\sigma L_{l^*}< 2(1-\theta \tau\sigma \lambda_{\max}(\vA\vA^\top)).\label{eq:optimal}\end{align} 
		For comparison, the condition in~\cite{vu2013splitting} is $\max(\tau,\sigma)\max(L_f,L_{l^*})< 2(1-\sqrt{\tau\sigma\|\vA\vA^\top\|})$. Our condition has two benefits. One is that we consider $\tau$ and $\sigma$ differently and can obtain a large stepsize even when the Lipschitz constants $L_f$ and $L_{l^*}$ have different scales. The other is the introduction of $\theta\in(3/4,1]$, which may increase the upper bounds for the stepsizes. The best result in this paper comes from choosing a $\theta$ that is close to $3/4$ even when $\theta=1$ is enough for $\vM_1$ being positive definite. See the example in Section~\ref{sec:optimal_bdd}.		

		\item (Special cases:) The positiveness of $\vM_1$ gives an upper bound for $\tau\sigma$ that depends on $\vP$, $\vD$ and $\vA$.
		The convergence of~\eqref{for:PD3O_nog} requires an upper bound for $\tau$ that is $\tau<2\beta$, see Theorem~\ref{thm:main}. 
		If $\nabla l^*$ is fixed for all $\vs$, e.g., problem~\eqref{for:main_problem_noInf}, then~\eqref{for:cocoer_l} is satisfied with any $\beta> 0$, and the upper bound of $\tau$ depends on $\vP$ and $L_f$ only, i.e., $\tau< 2\lambda_{\min}(\vP)L_f^{-1}$.
		The condition is strictly weaker than that in~\cite{vu2013splitting} and~\cite{boct2015convergence} because of the introduction of $\theta$.
		If $\nabla f$ is fixed for all $\vx$, e.g., the linear $f$ in linearized ALM, then~\eqref{for:cocoer_f} is satisfied with any $\beta> 0$, and the upper bound for $\tau$ depends on $\sigma$, $\vA$, $\vD$, $\vP$, and the Lipschitz constant of $\nabla l^*$ because of $\vM_1$ in~\eqref{for:cocoer_l}, i.e., $\sigma< 2\lambda_{\min}(\vD-(3/4)\tau\sigma\vA\vP^{-1}\vA^\top)L_{l^*}^{-1}$.
	\end{itemize}
\end{remark}

%


\begin{assumption} \label{asp:2} Let $(\vx^\star,\vs^\star)$ be any fixed point of~\eqref{for:PD3O_nog}. There exist $\mu_f\geq 0$, $\mu_h\geq 0$, and $\mu_l\geq 0$, such that, for any $\vx\in\cX$ and $\vs\in\cS$, 
	\begin{align}
	\langle  \vx-\vx^\star,\nabla f(\vx)-\nabla f(\vx^\star)\rangle \geq  &\mu_f\|\vx-\vx^\star\|_\vP^2,		 \label{for:monoto_f}\\
	\langle \vs-\vs^\star,\vp_h(\vs)-\vp_h(\vs^\star)\rangle \geq &\mu_{h}\|\vs-\vs^\star\|^2_{\vM_1},  \label{for:monoto_qh} \\
	\langle \vs-\vs^\star,\nabla l^*(\vs)-\nabla l^*(\vs^\star)\rangle \geq & \mu_{l}\|\vs-\vs^\star\|^2_{\vM_1},	 \label{for:monoto_l} 
	\end{align}
	where $\vp_h(\vs)\in\partial h^*(\vs)$ and $\vp_h(\vs^\star)\in\partial h^*(\vs^\star)$.  
\end{assumption}
The assumption is satisfied if functions $f(\vx)$, $h(\vs)$, and $l(\vs)$ are convex, and in this case, $\mu_f=\mu_{h}=\mu_{l}=0$. We choose the norms $\|\cdot\|_\vP$ and $\|\cdot\|_{\vM_1}$ for the two spaces for simplicity. All the results in this paper also hold for standard norms, but the formulas are complicated. We will need this assumption with positive values to show the linear convergence for strongly convex functions. In this case, because $\vP$ and $\vM_1$ are positive definite, $\mu_f>0$ (or $\mu_h>0$, $\mu_l>0$) is implied from the strong convexity of the function $f(\vx)$ (or $g^*(\vs)$, $l^*(\vs)$).

\subsection{Convergence for general convex functions}\label{sec:convergence_gen}
First of all, we find a subgradient of $h^*$ at $\vs^{k+1}$:
\begin{align}
\vq_h(\vs^{k+1}):=  {1 \over \tau}\vM\vs^k-{1\over\tau}\vM\vs^{k+1}+\vA\vx^{k+1}-\nabla l^*(\vs^k)\in \partial h^*(\vs^{k+1}). \label{def_qh1} 
\end{align}
It can be easily obtained from~\eqref{for:PD3O_nog}, and its proof is omitted here.
Let $(\vx^\star,\vs^\star)$ be any fixed point of~\eqref{for:PD3O_nog}, and we have a subgradient of $h^*$ at $\vs^\star$:
\begin{align}
\vq_h(\vs^\star):=    &  \vA\vx^\star-\nabla l^*(\vs^\star) \in \partial h^*(\vs^\star). \label{def_qh*}
\end{align}

\begin{lemma}[fundamental inequality]\label{lemma:fundamental_ineq}
	Let $(\vx^\star,\vs^\star)$ be any fixed point of~\eqref{for:PD3O_nog}, and $\{(\vx^k,\vs^k)\}$ a sequence generated by~\eqref{for:PD3O_nog}. 
	Then we have 
	\begin{align}
	& \|(\vx^{k+1},\vs^{k+1})-(\vx^\star,\vs^\star)\|_{\vP,\widetilde\vM}^2 \nonumber\\
	\leq	& \|(\vx^k,\vs^k)-(\vx^\star,\vs^\star)\|_{\vP,\widetilde\vM}^2  -  \|\vs^k-\vs^{k+1}\|_{\vM_1}^2 \nonumber\\
	& -2\tau\langle \vs^{k+1}-\vs^\star,\vq_h(\vs^{k+1})-\vq_h(\vs^\star)+\nabla l^*(\vs^k)-\nabla l^*(\vs^\star)\rangle\nonumber\\
	& +2\tau\langle \nabla f(\vx^k)-\nabla f(\vx^\star), \vx^\star-\vx^{k}+(4\theta-3)(\vx^k-\vx^{k+1})\rangle   \nonumber\\
	& -{(4\theta-3)\|\vx^k-\vx^{k+1}\|}_\vP^2+4(1-\theta)\tau^2\|\nabla f(\vx^k) -\nabla f(\vx^\star)\|_{\vP^{-1}}^2.\label{eqn:fundamental_ineq}
	\end{align}
\end{lemma}

\begin{proof}
	The definitions of $\vq_h(\vs^{k+1})$ and $\vq_h(\vs^\star)$ in~\eqref{def_qh1} and~\eqref{def_qh*}, respectively, and the update of $\vx^{k+1}$ in~\eqref{for:PD3O_nog_iteration_c} show
	\begin{align}
	& 2\tau\langle \vs^{k+1}-\vs^\star,\vq_h(\vs^{k+1})-\vq_h(\vs^\star)+\nabla l^*(\vs^k)-\nabla l^*(\vs^\star)\rangle \nonumber\\
	\overset{\eqref{def_qh1},\eqref{def_qh*}}{=} &  2\tau\langle \vs^{k+1}-\vs^\star,{1 \over \tau}\vM\vs^k-{1\over\tau}\vM\vs^{k+1}+\vA\vx^{k+1}-\vA\vx^\star\rangle \nonumber\\
	= & 2\langle \vs^{k+1}-\vs^\star,\vs^k-\vs^{k+1}\rangle_\vM +2\tau\langle \vs^{k+1}-\vs^\star,\vA\vx^{k+1}-\vA\vx^\star\rangle \nonumber\\
	= & 2\langle \vs^{k+1}-\vs^\star,\vs^k-\vs^{k+1}\rangle_\vM +2\tau\langle \vA^\top\vs^{k+1}-\vA^\top\vs^\star,\vx^{k+1}-\vx^\star\rangle \nonumber\\
	\overset{\eqref{for:PD3O_nog_iteration_c}}{=} & 2\langle \vs^{k+1}-\vs^\star,\vs^k-\vs^{k+1}\rangle_\vM +2\langle \vx^k-\vx^{k+1},\vx^{k+1}-\vx^\star\rangle_\vP \nonumber\\
	& -2\tau\langle \nabla f(\vx^k)-\nabla f(\vx^\star),\vx^{k+1}-\vx^\star\rangle \label{eqn:fundamental_0}\\
	= & \|\vs^{k}-\vs^\star\|_\vM^2-\|\vs^{k+1}-\vs^\star\|_\vM^2-\|\vs^k-\vs^{k+1}\|_\vM^2 \nonumber\\
	& +\|\vx^k-\vx^\star\|_\vP^2-\|\vx^{k+1}-\vx^\star\|_\vP^2-\|\vx^k-\vx^{k+1}\|_\vP^2 \nonumber\\
	& +2\tau\langle \nabla f(\vx^k)-\nabla f(\vx^\star),\vx^\star-\vx^{k+1}\rangle, \nonumber 
	\end{align}
	where we expanded the first two terms in~\eqref{eqn:fundamental_0} using $2\langle a,b\rangle =\|a+b\|^2-\|a\|^2-\|b\|^2$ to obtain the last equality. 
	Therefore, we have 
	\begin{align}
	& \|(\vx^{k+1},\vs^{k+1})-(\vx^\star,\vs^\star)\|_{\vP,\vM}^2 \nonumber\\
	= & 2\tau\langle \nabla f(\vx^k)-\nabla f(\vx^\star), \vx^\star-\vx^{k+1}\rangle \nonumber\\
	& -2\tau\langle \vs^{k+1}-\vs^\star,\vq_h(\vs^{k+1})-\vq_h(\vs^\star)+\nabla l^*(\vs^k)-\nabla l^*(\vs^\star)\rangle  \nonumber\\
	&  + \|(\vx^k,\vs^k)-(\vx^\star,\vs^\star)\|_{\vP,\vM}^2 - \|\vx^k-\vx^{k+1}\|_\vP^2 -  \|\vs^k-\vs^{k+1}\|_\vM^2. \label{eq:nonexpansive_cp_opt}
	\end{align}
	The fact that $\vM=\vM_1-\vM_2$ gives us an upper bound for the last term of~\eqref{eq:nonexpansive_cp_opt}.
	\begin{align}
	-  \|\vs^k-\vs^{k+1}\|_\vM^2 = & -  \|\vs^k-\vs^{k+1}\|_{\vM_1}^2 +  \|\vs^k-\vs^{k+1}\|_{\vM_2}^2 \nonumber\\
	=	     &   -  \|\vs^k-\vs^{k+1}\|_{\vM_1}^2+ \|\vs^k-\vs^\star+\vs^\star-\vs^{k+1}\|_{\vM_2}^2 \nonumber\\
	\leq   & -  \|\vs^k-\vs^{k+1}\|_{\vM_1}^2+2\|\vs^k-\vs^\star\|_{\vM_2}^2+2\|\vs^{k+1}-\vs^\star\|_{\vM_2}^2.\label{con_middle_term}
	\end{align} 
	Adding $2\|\vs^{k+1}-\vs^\star\|_{\vM_2}^2$ onto both sides of~\eqref{eq:nonexpansive_cp_opt}, recalling that $\widetilde\vM = \vM_1+\vM_2=\vM+2\vM_2$, and combining~\eqref{con_middle_term} and~\eqref{eq:nonexpansive_cp_opt}, we have  
	\begin{align}
	& \|(\vx^{k+1},\vs^{k+1})-(\vx^\star,\vs^\star)\|_{\vP,\widetilde\vM}^2 \nonumber\\
	\leq	& 2\tau\langle \nabla f(\vx^k)-\nabla f(\vx^\star), \vx^\star-\vx^{k+1}\rangle  \nonumber\\
	& -2\tau\langle \vs^{k+1}-\vs^\star,\vq_h(\vs^{k+1})-\vq_h(\vs^\star)+\nabla l^*(\vs^k)-\nabla l^*(\vs^\star)\rangle\nonumber\\
	& + \|(\vx^k,\vs^k)-(\vx^\star,\vs^\star)\|_{\vP,\widetilde\vM}^2 - \|\vx^k-\vx^{k+1}\|_\vP^2 -  \|\vs^k-\vs^{k+1}\|_{\vM_1}^2 \nonumber\\
	& +4\|\vs^{k+1}-\vs^\star\|_{\vM_2}^2. \label{eq:nonexpansive_cp_optt}
	\end{align}
	With the definition of $\vM_2$, the last term in~\eqref{eq:nonexpansive_cp_optt} can be written as
	\begin{align} 
	4\|\vs^{k+1}-\vs^\star\|_{\vM_2}^2 
	= & 4{(1-\theta)\|\tau\vP^{-1}\vA^\top\vs^{k+1}-\tau\vP^{-1}\vA^\top\vs^\star\|}_\vP^2 \nonumber\\
	=	& 4{(1-\theta)\|\vx^k-{\tau}\vP^{-1}\nabla f(\vx^k) -\vx^{k+1}+\tau\vP^{-1}\nabla f(\vx^\star)\|}_\vP^2  \nonumber\\
	= & 4{(1-\theta)\|\vx^k-\vx^{k+1}\|}_\vP^2+4(1-\theta)\tau^2\|\nabla f(\vx^k) -\nabla f(\vx^\star)\|_{\vP^{-1}}^2 \nonumber\\
	&-8(1-\theta)\tau\langle \vx^k-\vx^{k+1},\nabla f(\vx^k)-\nabla f(\vx^\star)\rangle \label{for:from_PD3O_nog_iteration_c},
	\end{align}
	where the second equality comes from~\eqref{for:PD3O_nog_iteration_c}. Then, we plug~\eqref{for:from_PD3O_nog_iteration_c} into~\eqref{eq:nonexpansive_cp_optt} and obtain
	\begin{align}
	& \|(\vx^{k+1},\vs^{k+1})-(\vx^\star,\vs^\star)\|_{\vP,\widetilde\vM}^2 \nonumber\\
	\leq	& 2\tau\langle \nabla f(\vx^k)-\nabla f(\vx^\star), \vx^\star-\vx^{k}+(4\theta-3)(\vx^k-\vx^{k+1})\rangle  \nonumber\\
	& -2\tau\langle \vs^{k+1}-\vs^\star,\vq_h(\vs^{k+1})-\vq_h(\vs^\star)+\nabla l^*(\vs^k)-\nabla l^*(\vs^\star)\rangle\nonumber\\
	& + \|(\vx^k,\vs^k)-(\vx^\star,\vs^\star)\|_{\vP,\widetilde\vM}^2  -  \|\vs^k-\vs^{k+1}\|_{\vM_1}^2 \nonumber\\
	& -{(4\theta-3)\|\vx^k-\vx^{k+1}\|}_\vP^2+4(1-\theta)\tau^2\|\nabla f(\vx^k) -\nabla f(\vx^\star)\|_{\vP^{-1}}^2.\nonumber 
	\end{align}
	The result is proved. \qed
\end{proof}

\begin{lemma}\label{lem:main}
	Let~\eqref{for:cocoer_l} be satisfied, then
	\begin{align*}
	&  -\|\vs^k-\vs^{k+1}\|_{\vM_1}^2 -2\tau\langle\vs^{k+1}-\vs^\star,\nabla l^*(\vs^k)-\nabla l^*(\vs^\star)\rangle \nonumber\\
	\leq & -\left(1-\tau / (2\beta)\right)\|\vs^k-\vs^{k+1}\|_{\vM_1}^2.
	\end{align*}	
\end{lemma}	
\begin{proof}
	Because $\vM_1$ is positive definite, we have
	\begin{align*}
	& -\|\vs^k-\vs^{k+1}\|_{\vM_1}^2 -2\tau\langle\vs^{k+1}-\vs^\star,\nabla l^*(\vs^k)-\nabla l^*(\vs^\star)\rangle \nonumber\\
	=  & -\|\vs^k-\vs^{k+1}\|_{\vM_1}^2 -2\tau\langle\vs^{k+1}-\vs^k,\nabla l^*(\vs^k)-\nabla l^*(\vs^\star)\rangle \nonumber\\
	& -2\tau\langle\vs^k-\vs^\star,\nabla l^*(\vs^k)-\nabla l^*(\vs^\star)\rangle\nonumber\\
	\leq  & -\|\vs^k-\vs^{k+1}\|_{\vM_1}^2 +{\tau\over2\beta}\|\vs^k-\vs^{k+1}\|_{\vM_1}^2+{2\tau\beta}\|\nabla l^*(\vs^k)-\nabla l^*(\vs^\star)\|_{\vM_1^{-1}}^2 \nonumber\\
	& -2\tau\beta\|\nabla l^*(\vs^k)-\nabla l^*(\vs^\star)\|_{\vM_1^{-1}}^2	\\
	=  & -\|\vs^k-\vs^{k+1}\|_{\vM_1}^2 +{\tau\over2\beta}\|\vs^k-\vs^{k+1}\|_{\vM_1}^2,
	\end{align*}
	where the inequality comes from the Cauchy-Schwarz inequality and~\eqref{for:cocoer_l}. \qed
\end{proof}

\begin{theorem}\label{thm:main}	Let Assumption~\ref{asp:1} hold, $\theta\in(3/4,1]$, and $\tau\in(0,2\beta)$. 
	The sequence $\{(\vx^k,\vs^k)\}$ is generated by~\eqref{for:PD3O_nog}. 
	For any fixed point $(\vx^\star,\vs^\star)$ of~\eqref{for:PD3O_nog}, we have
	\begin{align}
	& \|(\vx^{k+1},\vs^{k+1})-(\vx^\star,\vs^\star)\|_{\vP,\widetilde\vM}^2- \|(\vx^k,\vs^k)-(\vx^\star,\vs^\star)\|_{\vP,\widetilde\vM}^2 \nonumber\\
	\leq &  -\left(1-{\tau \over 2\beta}\right)\|\vs^k-\vs^{k+1}\|_{\vM_1}^2 -{(4\theta-3)(2\beta-\tau)\over2\beta-4(1-\theta)\tau }\|\vx^k-\vx^{k+1}\|_\vP^2. \label{eqn:bounded}
	\end{align}
\end{theorem}

\begin{proof}
	Applying Lemma~\ref{lem:main} and $h$ being convex to the inequality~\eqref{eqn:fundamental_ineq} in Lemma~\ref{lemma:fundamental_ineq} gives 
	\begin{align}
	& \|(\vx^{k+1},\vs^{k+1})-(\vx^\star,\vs^\star)\|_{\vP,\widetilde\vM}^2 \nonumber\\
	\leq&  \|(\vx^k,\vs^k)-(\vx^\star,\vs^\star)\|_{\vP,\widetilde\vM}^2-\left(1-\tau /(2\beta)\right)\|\vs^k-\vs^{k+1}\|_{\vM_1}^2  \nonumber\\
	& + \underbrace{2\tau\langle \nabla f(\vx^k)-\nabla f(\vx^\star), \vx^\star-\vx^{k}\rangle}_{A} + 4(1-\theta)\tau^2\|\nabla f(\vx^k) -\nabla f(\vx^\star)\|_{\vP^{-1}}^2  \nonumber\\
	&-{(4\theta-3)\|\vx^k-\vx^{k+1}\|}_\vP^2+\underbrace{2\tau(4\theta-3)\langle \nabla f(\vx^k)-\nabla f(\vx^\star) ,\vx^k-\vx^{k+1}\rangle}_{B}. \label{proof_thm1}
	\end{align}
	Next we bound terms A and B. For term A, the assumption~\eqref{for:cocoer_f} implies
	\begin{align}
	2\tau\langle \nabla f(\vx^k)-\nabla f(\vx^\star), \vx^\star-\vx^{k}\rangle \leq -2\tau\beta \|\nabla f(\vx^k) -\nabla f(\vx^\star)\|_{\vP^{-1}}^2, \label{proof_termA}
	\end{align}
	and the Cauchy-Schwarz inequality applied to term B implies 
	\begin{align}
	& 2\tau(4\theta-3)\langle \nabla f(\vx^k)-\nabla f(\vx^\star) ,\vx^k-\vx^{k+1}\rangle \nonumber\\
	\leq & (2\tau\beta-4(1-\theta)\tau^2)\|\nabla f(\vx^k)-\nabla f(\vx^\star)\|_{\vP^{-1}}^2 \nonumber\\
	&  + {\tau(4\theta-3)^2\over 2\beta-4(1-\theta)\tau}\|\vx^k-\vx^{k+1}\|_\vP^2, \label{proof_termB}
	\end{align}
	when $\theta\in(3/4,1]$ and $\tau\in(0,2\beta)$. The inequality holds because $2\beta-4(1-\theta)\tau>0$, owing to the bounds on $\tau$ and $\theta$.
	Plugging~\eqref{proof_termA} and~\eqref{proof_termB} into~\eqref{proof_thm1}, we have
	\begin{align*}
	& \|(\vx^{k+1},\vs^{k+1})-(\vx^\star,\vs^\star)\|_{\vP,\widetilde\vM}^2 \nonumber\\
	\leq&  \|(\vx^k,\vs^k)-(\vx^\star,\vs^\star)\|_{\vP,\widetilde\vM}^2-\left(1-\tau / (2\beta)\right)\|\vs^k-\vs^{k+1}\|_{\vM_1}^2  \nonumber\\
	&  -{(4\theta-3)\|\vx^k-\vx^{k+1}\|}_\vP^2  + {\tau(4\theta-3)^2\over 2\beta-4(1-\theta)\tau}\|\vx^k-\vx^{k+1}\|_\vP^2 \nonumber\\
	= &  \|(\vx^k,\vs^k)-(\vx^\star,\vs^\star)\|_{\vP,\widetilde\vM}^2 -\left(1-\tau /(2\beta)\right)\|\vs^k-\vs^{k+1}\|_{\vM_1}^2 \nonumber\\
	&  -{(4\theta-3)(2\beta-\tau)\over2\beta-4(1-\theta)\tau }\|\vx^k-\vx^{k+1}\|_\vP^2. 
	\end{align*}
	The inequality~\eqref{eqn:bounded} is proved. \qed
\end{proof}

\begin{remark}\label{remark:PIDP_ALM}
	When $\beta=+\infty$, i.e., the Lipschitz constant of $\nabla f(\vx)$ and $\nabla l^*(\vs)$ is 0, then~\eqref{eqn:bounded} becomes
	\begin{align*}
	& \|(\vx^{k+1},\vs^{k+1})-(\vx^\star,\vs^\star)\|_{\vP,\widetilde\vM}^2- \|(\vx^k,\vs^k)-(\vx^\star,\vs^\star)\|_{\vP,\widetilde\vM}^2 \nonumber\\
	\leq  & -\|\vs^k-\vs^{k+1}\|_{\vM_1}^2 -(4\theta-3)\|\vx^k-\vx^{k+1}\|_\vP^2. 
	\end{align*}
	This is the key result in~\cite[Theorem 3.1]{he2016positive} for linearized ALM. 
	In~\cite{he2016positive}, the authors also considered the case with a general dual stepsize. 
\end{remark}

\begin{remark}[Large stepsizes] 
	We let $\vP=\vI$ and $\vD=\vI$ for simplicity. Consider the problem~\eqref{for:main_problem_noInf} without function $l$. We have $\beta=1/L$, where $L$ is the Lipschitz constant of $\nabla f$. Then we can choose $\tau<2/L$, and $\tau\sigma\leq 4/(3\|\vA\vA^\top\|)$. 
	
	However, for the problem~\eqref{for:main_problem} with function $l$, the choice of the primal stepsize $\tau$ also depends on $\sigma$ because of the operator $\vM_1$ in the assumption~\eqref{for:cocoer_l}. For this case, how to choose $\tau$ and $\sigma$ is complicated. From Remark~\ref{remark1}, if $f$ and $l^*$ have Lipschitz continuous gradients with constants $L_f$ and $L_{l^*}$, respectively, a sufficient condition for convergence is $\tau L_f< 2$ and $\sigma L_{l^*}< 2(1-(3/4) \tau\sigma \|\vA\vA^\top\|)$. Except the same conditions $\tau<2/L$ and  $\tau\sigma\leq 4/(3\|\vA\vA^\top\|)$, there is an additional condition $\sigma< 2(1-(3/4) \tau\sigma \|\vA\vA^\top\|)/ L_{l^*}$.


	%
	%
\end{remark}

\begin{theorem}\label{thm:main2}
	Under the assumptions in Theorem \ref{thm:main}, the sequence $\{(\vx^k,\vs^k)\}$  converges weakly to a fixed point of~\eqref{for:PD3O_nog}.  
	If the iteration~\eqref{for:PD3O_nog} is demicompact at $\vzero$~\cite{petryshyn1966construction}\footnote{ An operator $\vT$ is demicompact at $\vx\in\cH$ if for every bounded sequence $\{\vx^k\}_{k\geq0}$ in $\cH$ such that $T\vx^k - \vx^k\rightarrow \vx$, there exists a strongly convergent subsequence.}, the sequence converges strongly.
\end{theorem}

\begin{proof}
Theorem~\ref{thm:main} shows that the sequence $\{(\vx^k,\vs^k)\}$ is bounded, so weakly convergent subsequences of $\{(\vx^k,\vs^k)\}$ exist. For any weakly convergent subsequence such that $(\vx^{k_i},\vs^{k_i})\rightharpoonup (\vx,\vs)$, the inequality~\eqref{eqn:bounded} gives $(\vx^{k_i-1}-\vx^{k_i},\vs^{{k_i}-1}-\vs^{k_i})\rightarrow 0$. Then based on the iteration~\eqref{for:PD3O_nog}, we obtain~\cite[Fact 1.37]{bauschke2011convex} 
	\begin{align*}
	\nabla f(\vx^{k_i})+\vA^\top\vs^{k_i} = 	{1\over \tau}\vP(\vx^{k_i-1}-\vx^{k_i}) +\nabla f(\vx^{k_i})-\nabla f(\vx^{k_i-1})\rightarrow \vzero,\\
	-\vA\vx^{k_i}+\vq_h(\vs^{k_i}) + \nabla l^*(\vs^{k_i}) = {1 \over \tau}\vM(\vs^{k_i-1}-\vs^{k_i})-\nabla l^*(\vs^{k_i-1})+ \nabla l^*(\vs^{k_i}) \rightarrow \vzero.
	\end{align*}
	Because $f$, $h^*$, and $l^*$ are convex, the operator 
	\begin{align*}
	 \begin{bmatrix}
	\nabla f  & \vA^\top \\-\vA & \partial h^* +\nabla l^*
	\end{bmatrix}
\end{align*}
is maximal monotone. 
	Thus, $(\vx,\vs)$ is a fixed point of~\eqref{for:PD3O_nog} because of~\cite[Proposition 20.33(ii)]{bauschke2011convex}.  

	The inequality~\eqref{eqn:bounded} also shows that the sequence $\{(\vx^k,\vs^k)\}$ is Fej\'er monotone with respect to the set of fixed points of~\eqref{for:PD3O_nog}. Then~\cite[Theorem 5.5]{bauschke2011convex} shows that $\{(\vx^{k},\vs^{k})\}$ converges weakly to a fixed point of~\eqref{for:PD3O_nog}. 
	
	The inequality~\eqref{eqn:bounded} shows that $\{(\vx^k,\vs^k)\}$ is a bounded sequence and $(\vx^{k+1}-\vx^k,\vs^{k+1}-\vs^k)\rightarrow 0$. Then the demicompactness of the iteration in~\eqref{for:PD3O_nog} at $\vzero$ shows that there is a strongly convergent subsequence $(\vx^{k_n},\vs^{k_n})\rightarrow (\bar\vx^\star,\bar\vs^\star)$, and $(\bar\vx^\star,\bar\vs^\star)$ is a fixed point of~\eqref{for:PD3O_nog} because this subsequence is also weakly convergent. Then the inequality \eqref{eqn:bounded} shows that the whole sequence $\{(\vx^k,\vs^k)\}$ converges to the fixed point $(\bar\vx^\star,\bar\vs^\star)$.
	\qed
\end{proof}

\begin{remark} When $\cX$ and $\cS$ are finite dimensional, the sequence $\{(\vx^k,\vs^k)\}$  converges strongly to a fixed point of~\eqref{for:PD3O_nog}.
\end{remark}

In Theorem~\ref{thm:main2}, we showed the convergence of this primal-dual algorithm without providing the convergence rate. The ergodic sublinear convergence rate is showed for primal-dual algorithms for more general problems~\cite{chambolle2016ergodic,yan2017three}.

\subsection{Linear convergence}\label{sec:linear_conv}
In this subsection,  we prove the linear convergence of the sequence $\{(\vx^k,\vs^k)\}$ in Theorem~\ref{thm:main_linear} under the additional Assumption~\ref{asp:2}. 

Before showing the linear convergence, we prove the following lemma, which provides a different upper bound for the same object in Lemma~\ref{lem:main}.
\begin{lemma}\label{lem:linear_main}
	Let~\eqref{for:cocoer_l} and~\eqref{for:monoto_l} be satisfied, then
	\begin{align}\label{eqn:linear_main}
	& -\|\vs^k-\vs^{k+1}\|_{\vM_1}^2 -2\tau\langle\vs^{k+1}-\vs^\star,\nabla l^*(\vs^k)-\nabla l^*(\vs^\star)\rangle \nonumber\\
	\leq  &-\|\vM_2(\vs^{k+1}-\vs^k)+\tau\vA\vx^{k+1}-\tau\vA\vx^\star-\tau\vq_h(\vs^{k+1})+\tau\vq_h(\vs^\star)\|_{\vM_1^{-1}}^2\\
	&		 -\left(2\tau-\tau^2/\beta\right)\mu_{l}\|\vs^k-\vs^{*}\|_{\vM_1}^2\nonumber.
	\end{align}	
\end{lemma}	
\begin{proof}
Because $\vM_1$ is positive definite, we have
\begin{align}
& -\|\vs^k-\vs^{k+1}\|_{\vM_1}^2 -2\tau\langle\vs^{k+1}-\vs^\star,\nabla l^*(\vs^k)-\nabla l^*(\vs^\star)\rangle \nonumber\\
=  & -\|\vs^k-\vs^{k+1}\|_{\vM_1}^2 -2\tau\langle\vM_1^{1/2}(\vs^{k+1}-\vs^k),\vM_1^{-1/2}(\nabla l^*(\vs^k)-\nabla l^*(\vs^\star))\rangle\nonumber\\
& -2\tau\langle\vs^k-\vs^\star,\nabla l^*(\vs^k)-\nabla l^*(\vs^\star)\rangle\nonumber\\
=  & -\|\vM_1^{1/2}(\vs^{k+1}-\vs^k)+\vM_1^{-1/2}\tau(\nabla l^*(\vs^k)-\nabla l^*(\vs^\star))\|^2\nonumber\\
& +\tau^2\|\nabla l^*(\vs^k)-\nabla l^*(\vs^\star)\|_{\vM_1^{-1}}^2-2\tau\langle\vs^k-\vs^\star,\nabla l^*(\vs^k)-\nabla l^*(\vs^\star)\rangle. \label{lemma3:a}
\end{align}
The first term on the right-hand side of~\eqref{lemma3:a} becomes
\begin{align}
&  -\|\vM_1^{1/2}(\vs^{k+1}-\vs^k)+\vM_1^{-1/2}\tau(\nabla l^*(\vs^k)-\nabla l^*(\vs^\star))\|^2 \nonumber\\
= &  -\|\vM_1(\vs^{k+1}-\vs^k)+\tau(\nabla l^*(\vs^k)-\nabla l^*(\vs^\star))\|_{\vM_1^{-1}}^2 \nonumber\\
= &  -\|\vM_2(\vs^{k+1}-\vs^k)+\vM(\vs^{k+1}-\vs^k)+\tau(\nabla l^*(\vs^k)-\nabla l^*(\vs^\star))\|_{\vM_1^{-1}}^2\nonumber\\
\overset{\eqref{def_qh1},\eqref{def_qh*}}{=} &  -\|\vM_2(\vs^{k+1}-\vs^k)+\tau\vA\vx^{k+1}-\tau\vA\vx^\star-\tau\vq_h(\vs^{k+1})+\tau\vq_h(\vs^\star)\|_{\vM_1^{-1}}^2, \nonumber
\end{align}
where the second equality comes from $\vM=\vM_1-\vM_2$.

For the other two terms on the right-hand side of~\eqref{lemma3:a}, we have
\begin{align}
& \tau^2\|\nabla l^*(\vs^k)-\nabla l^*(\vs^\star)\|_{\vM_1^{-1}}^2-2\tau\langle\vs^k-\vs^\star,\nabla l^*(\vs^k)-\nabla l^*(\vs^\star)\rangle  \nonumber\\
\overset{\eqref{for:cocoer_l},\eqref{for:monoto_l}}{\leq} &-(2\tau-\tau^2/\beta)\mu_l\|\vs^k-\vs^\star\|_{\vM_1}^2. \nonumber
\end{align}
Combining both inequalities together with~\eqref{lemma3:a} gives~\eqref{eqn:linear_main}.	\qed
\end{proof}

\begin{theorem} \label{thm:main_linear}
	Let $(\vx^\star,\vs^\star)$ be a fixed point  of~\eqref{for:PD3O_nog} and Assumptions~\ref{asp:1} and~\ref{asp:2} hold. 
	Define $\widehat\vM   :=  (1+ 2\tau\mu_h)\vM_1+\vM_2$, and we have
	\begin{align}
	\|(\vx^{k+1},\vs^{k+1})-(\vx^\star,\vs^\star)\|_{\vP,\widehat\vM}^2 
	\leq	 
	\rho_1\|(\vx^k,\vs^k)-(\vx^\star,\vs^\star)\|_{\vP,\widehat\vM}^2, \label{eqn:linear1}\end{align}
	where
	\begin{align*}
	\rho_1 = \max\left( {{1- (2\tau-\tau^{2}/\beta)\mu_{l}+C_1} \over {1 +2\tau\mu_h}+C_1},  1- (2\tau-\tau^{2}/\beta)\mu_{f}\right).
	\end{align*}
	Here $C_1\equiv\|\vM_1^{-1/2}\vM_2\vM_1^{-1/2}\|\geq 0$.
	The sequence $\{(\vx^k,\vs^k)\}$ converges linearly to the fixed point $(\vx^\star,\vs^\star)$ with rate $\rho_1<1$ if $\tau\in(0,2\beta)$, $\mu_h+\mu_l>0$, and $\mu_f>0$.
\end{theorem}

\begin{proof}
	Applying Lemma~\ref{lem:linear_main} to~\eqref{eqn:fundamental_ineq} in Lemma~\ref{lemma:fundamental_ineq} gives
	\begin{align}
	& \|(\vx^{k+1},\vs^{k+1})-(\vx^\star,\vs^\star)\|_{\vP,\widetilde\vM}^2 \nonumber\\
	\leq	& \|(\vx^k,\vs^k)-(\vx^\star,\vs^\star)\|_{\vP,\widetilde\vM}^2  -  \left(2\tau-\tau^2/ \beta\right)\mu_{l}\|\vs^k-\vs^{*}\|_{\vM_1}^2 \nonumber\\
	& -2\tau\langle \vs^{k+1}-\vs^\star,\vq_h(\vs^{k+1})-\vq_h(\vs^\star)\rangle\nonumber\\
	& +2\tau\langle \nabla f(\vx^k)-\nabla f(\vx^\star), \vx^\star-\vx^{k}+(4\theta-3)(\vx^k-\vx^{k+1})\rangle   \nonumber\\
	& -{(4\theta-3)\|\vx^k-\vx^{k+1}\|}_\vP^2+4(1-\theta)\tau^2\|\nabla f(\vx^k) -\nabla f(\vx^\star)\|_{\vP^{-1}}^2\nonumber\\
	= & \|(\vx^k,\vs^k)-(\vx^\star,\vs^\star)\|_{\vP,\widetilde\vM}^2  -  \left(2\tau-\tau^2/ \beta\right)\mu_{l}\|\vs^k-\vs^{*}\|_{\vM_1}^2 \nonumber\\
	& -2\tau\langle \vs^{k+1}-\vs^\star,\vq_h(\vs^{k+1})-\vq_h(\vs^\star)\rangle\nonumber\\
	& -2\tau\langle \nabla f(\vx^k)-\nabla f(\vx^\star), \vx^k-\vx^{*}\rangle  +\tau^2\|\nabla f(\vx^k) -\nabla f(\vx^\star)\|_{\vP^{-1}}^2	\nonumber\\
	& -(4\theta-3)\|\vx^k-\vx^{k+1}-\tau\vP^{-1}(\nabla f(\vx^k)-\nabla f(\vx^\star))\|^2_{\vP}. \nonumber
	\end{align}
	Note that
	\begin{align*}
	& -2\tau\langle \nabla f(\vx^k)-\nabla f(\vx^\star), \vx^k-\vx^{*}\rangle  +\tau^2\|\nabla f(\vx^k) -\nabla f(\vx^\star)\|_{\vP^{-1}}^2	\\
	\overset{\eqref{for:cocoer_f}}{\leq} & -(2\tau-\tau^2/\beta)\langle \nabla f(\vx^k)-\nabla f(\vx^\star), \vx^k-\vx^{*}\rangle \\
	\overset{\eqref{for:monoto_f}}{\leq} & -(2\tau-\tau^2/\beta)\mu_f\|\vx^k-\vx^{*}\|_\vP^2.
	\end{align*}
	Then we have, together with~\eqref{for:monoto_qh},
	\begin{align*}
	& \|(\vx^{k+1},\vs^{k+1})-(\vx^\star,\vs^\star)\|_{\vP,\widetilde\vM}^2 \nonumber\\
	\leq & \|(\vx^k,\vs^k)-(\vx^\star,\vs^\star)\|_{\vP,\widetilde\vM}^2  -  \left(2\tau-\tau^2/ \beta\right)\mu_{l}\|\vs^k-\vs^{*}\|_{\vM_1}^2 \nonumber\\
	& -2\tau\mu_h\|\vs^{k+1}-\vs^\star\|^2_{\vM_1} -(2\tau-\tau^2/\beta)\mu_f\|\vx^k-\vx^{*}\|_\vP^2.
	\end{align*}
	That is 
	\begin{align}
	& \|\vx^{k+1}-\vx^\star\|_{\vP}^2 + \|\vs^{k+1}-\vs^\star\|_{(1+2\tau\mu_h)\vM_1+\vM_2}^2 \nonumber\\
	\leq & (1-(2\tau-\tau^2/\beta)\mu_f)\|\vx^k-\vx^\star\|_{\vP}^2 +\|\vs^k-\vs^\star\|_{(1-  \left(2\tau-\tau^2/\beta\right)\mu_{l})\vM_1+\vM_2}^2. \label{eqn:thm3_t1}
	\end{align}
	For the last term on the right hand of~\eqref{eqn:thm3_t1}, we have 
	\begin{align*}
	& \|\vs^k-\vs^\star\|_{(1-  \left(2\tau-\tau^2/\beta\right)\mu_{l})\vM_1+\vM_2}^2 \nonumber\\
	=    & \|\vM_1^{1/2}(\vs^k-\vs^\star)\|_{(1-  \left(2\tau-\tau^2/\beta\right)\mu_{l})\vI+\vM_1^{-1/2}\vM_2\vM_1^{-1/2}}^2\\
	\leq &  {{1- (2\tau-\tau^{2}/\beta)\mu_{l}+C_1} \over {1 +2\tau\mu_h}+C_1}\|\vM_1^{1/2}(\vs^k-\vs^\star)\|_{(1+2\tau\mu_{h})\vI+\vM_1^{-1/2}\vM_2\vM_1^{-1/2}}^2\\
	=    &  {{1- (2\tau-\tau^{2}/\beta)\mu_{l}+C_1} \over {1 +2\tau\mu_h}+C_1}\|\vs^k-\vs^\star\|_{(1+2\tau\mu_{h})\vM_1+\vM_2}^2.
	\end{align*}
	Therefore, the inequality~\eqref{eqn:linear1} is proved. \qed
\end{proof}

Note that paper~\cite{chen2013primal} proves the linear convergence rate for the case with $l^*(\vs)\equiv0$ and $\vM_2=0$ as
\begin{align*}
\max\left( 1-{\min_{\vx:\|\vx\|=1}\|\vA\vA^\top\vx\|\over\|\vA\vA^\top\|},  1- (2\tau-\tau^{2}/\beta)\mu_{f}\right)
\end{align*}
under the additional assumption that $\vA\vA^\top$ is surjective. 
However, $\mu_h>0$ is not required.

Next, we compare this result with the linear convergence rate of Condat-Vu in~\cite{boct2015convergence} by letting $h^*(\vs)\equiv 0$ and $\vM_2=\vzero$, $\vP=\vI$, $\vD=\vI$. For simplicity, we assume that $f$ and $l^*$ are both $\mu$-strongly convex and have $L$-Lipschitz continuous gradients. The linear convergence rate of Condat-Vu  is 
$${4\over 4+\min\left({\mu^2\over L^2},\sqrt{\mu^2\over \|\vA\vA^\top\|}\right)},$$
with the primal and dual stepsizes in the order of ${\mu/L^2}$. However, if we let $\tau=\sigma$ in~\eqref{for:PD3O_nog}, then we have $\mu_l\geq \mu =\mu_f$ and $\beta = (1-\tau^2\|\vA\vA^\top\|)/L$ in Assumptions~\ref{asp:1} and~\ref{asp:2}. In addition, we let $\tau=\beta$, then the linear convergence rate in  Theorem~\ref{thm:main_linear} becomes 
$$1-\tau \mu =1-{2\mu\over\sqrt{L^2+4\|\vA\vA^\top\|}+L}.$$
We can see that the linear convergence rate of~\eqref{for:PD3O_nog} is much better than that of Condata-Vu in~\cite{boct2015convergence}.

\subsection{Tight upper bound for the stepsizes}\label{sec:optimal_bdd}
A very simple example was provided in~\cite{he2016positive} to show the upper bound's tightness for a case without infimal convolution. In this subsection, we provide another example to show the tightness for a case with infimal convolution. This result will be applied to decentralized consensus optimization in the next section. Given a self-adjoint positive definite  operator $\vD$, we consider the following optimization problem:
$$\Min_\vx~{\va^\top\vx}+{\vx^\top\vA^\top\vD^{-1}\vA\vx\over2}.$$  
It is a special case of~\eqref{for:main_problem} with $f(\vx)=\va^\top\vx$, $h^*(\vy)=0$, and $l^*(\vy)=\vy^\top\vD\vy/2$. The primal-dual iteration~\eqref{for:PD3O_nog} after a change of order is 
\begin{align*}
\vx^{k+1} = ~&\vx^k-\tau\vP^{-1}\va-\tau \vP^{-1}\vA^\top\vs^{k},\\
\vs^{k+1} = ~&(\vI-\tau\sigma\vD^{-1}\vA\vP^{-1}\vA^\top-\sigma\vI)\vs^k+\sigma\vD^{-1}\vA \vx^{k+1}-\tau\sigma\vD^{-1}\vA\vP^{-1}\va.
\end{align*}
Denote $\widetilde\vD=\vD^{-1/2}\vA\vP^{-1}\vA^\top\vD^{-1/2}$. Then the iteration is equivalent to 
	\begin{align*}
	\begin{bmatrix}
	\vD^{-1/2}\vA\vx^{k+1}\\
	\vD^{1/2}\vs^{k+1}
	\end{bmatrix}= &\begin{bmatrix}
	\vI &  -\tau\widetilde\vD\\
	\sigma\vI & (1-\sigma)\vI-2\tau\sigma\widetilde{\vD}
	\end{bmatrix}\begin{bmatrix}
	\vD^{-1/2}\vA\vx^k\\
	\vD^{1/2}\vs^{k}
	\end{bmatrix}\\
	&-\begin{bmatrix}
	\tau{\vD}^{-1/2}\vA\vP^{-1}\va\\	2\tau\sigma{\vD}^{-1/2}\vA\vP^{-1}\va
	\end{bmatrix}.
	\end{align*}
The convergence of this iteration for any given initial $(\vs^0,\vx^0)$ requires the magnitudes of the eigenvalues of the operator 
	$$\begin{bmatrix}
	\vI &  -\tau\widetilde\vD\\
	\sigma\vI & (1-\sigma)\vI-2\tau\sigma\widetilde{\vD}
	\end{bmatrix}$$
being less than 1. Since $\widetilde{\vD}$ is self-adjoint, we need the magnitudes of the eigenvalues of 
	$$\widetilde\vM:=\left[\begin{array}{cc}
	1 &  -\tau\lambda\\
	\sigma & 1 -\sigma-2\tau\sigma\lambda
	\end{array}\right]$$
being less than 1 for all $\lambda$ being the eigenvalues of $\widetilde{\vD}$. We calculate the determinant of $\widetilde\vM-d\vI$ for any $d$ below:
\begin{align*}
\det(\widetilde\vM-d\vI) = d^2-(2-\sigma-2\tau\sigma\lambda)d+(1-\sigma-\tau\sigma\lambda).
\end{align*}
Particularly, the convergence requires $\det(\widetilde\vM+\vI)>0$, that is 
\begin{align*}
1+(2-\sigma-2\tau\sigma\lambda)+(1-\sigma-\tau\sigma\lambda)=4-3\tau\sigma\lambda -2\sigma>0.
\end{align*}
It is equivalent to 
$$\sigma < 2\left({1}-{3\over 4}\tau\sigma\|\widetilde{\vD}\|\right).$$
On the other hand, we proved the convergence of the primal-dual algorithm under the condition
$$\tau< 2\beta = 2\lambda_{\min}(\vD^{-1/2}\vM_1\vD^{-1/2})={2\tau\over\sigma}(1-\theta\tau\sigma\|\widetilde{\vD}\|)$$
for some $\theta\in(3/4,1]$. It shows that the upper bounds for the stepsizes in this paper are optimal.

	\section{Application in decentralized consensus optimization}\label{sec:compare}
	In this section, we first show that algorithm~\eqref{for:PD3O_nog} recovers PG-EXTRA~\cite{shi2015proximal} for decentralized consensus optimization. Then we provide its convergence result under a weaker condition than that in~\cite{shi2015proximal} and a tight upper bound for the stepsize. Note that PG-EXTRA was shown to be equivalent to Condat-Vu for a problem without infimal convolution~\cite{wu2018decentralized}, but this equivalence can not give the weaker condition for convergence and the tight upper bound for the stepsize.
	
	We use the same notation as~\cite{shi2015proximal}. The decentralized consensus problem is 
	\begin{align*}
	\Min_{x\in\vR^p}~\sum_{i=1}^{n}s_{i}(x)+r_{i}(x),
	\end{align*}
	where $s_i:\vR^p\rightarrow\vR$ and $r_i:\vR^p\rightarrow (-\infty,+\infty]$ are proper lsc convex functions held privately by the node $i$ to encode the node's objective function. The objective of decentralized consensus is minimizing the sum of all private objective functions while using information exchange between neighboring nodes in a network. Here $s_i$ has a Lipschitz continuous gradient with parameter $L>0$ and the proximal mapping of $r_i$ is simple.	We let $x_i$ be one copy of $x$ kept at node $i$. These $\{x_i\}_{i=1}^n$ are not the same in general, and we say that it is consensual if they are the same. Stacking all the copies together, we define 
	\begin{equation*}
	\vx:=
	\begin{pmatrix}
	-& x_1^\top &- \\ 
	-	& x_2^\top &- \\ 
	& \vdots & \\ 
	-& x_n^\top &- 
	\end{pmatrix}
	\in \vR^{n\times p},
	\end{equation*}
	and
	\begin{align*}
	{s}(\vx)=\sum_{i=1}^{n}s_{i}(x_i),\quad {r}(\vx)=\sum_{i=1}^{n}r_{i}(x_i).
	\end{align*}
	Then the decentralized consensus problem becomes 
	\begin{align*}
	\Min\limits_{\vx}~&s(\vx)+r(\vx),
	\mbox{ subject to } x_1=x_2=\cdots=x_n.
	\end{align*}
	The gradient of $s$ at $\vx$ is written in the following matrix form:
	\begin{equation*}
	\nabla s (\vx):=
	\begin{pmatrix}
	-& \left(\nabla s_{1}(x_{1})\right)^{\top} &- \\ 
	-	& \left(\nabla s_{2}(x_{2})\right)^{\top} &- \\ 
	& \vdots & \\ 
	-& \left(\nabla s_{n}(x_{n})\right)^{\top} &- 
	\end{pmatrix}
	\in \vR^{n\times p},
	\end{equation*}
	and $\|\cdot\|_F$ is the Frobenius norm for a matrix in $\vR^{n\times p}$. 
	One iteration of PG-EXTRA reads as
	\begin{subequations}\label{for:pg_extra}
		\begin{align}
		\vz^{k+1} & = \vz^{k}-\vx^k+{\vI+\vW\over 2}(2\vx^k-\vx^{k-1})-\alpha\nabla s(\vx^k)+\alpha\nabla s(\vx^{k-1}), \label{for:pg_extra_a}\\
		\vx^{k+1} & = \argmin\limits_{\vx}~r(\vx)+\frac{1}{2\alpha}\|\vx-\vz^{k+1}\|_F^{2},
		\end{align}
	\end{subequations}
	where $\alpha$ is the stepsize and $\vW$ is a symmetric matrix that represents information exchange between neighboring nodes. We have $\vI-\vW$ being positive semidefinite, so we can find $\vA$ such that $\vI-\vW=\vA\vA^\top$. In addition, we assume that $\Null(\vA^\top)=\Null(\vI-\vW)=\Span(\vone_{n\times 1})$, which means that $\vA^\top\vx=\vzero$ is equivalent to $x_1=x_2=\cdots=x_n$. Therefore, the decentralized consensus problem becomes 
	\begin{align*}
	\Min\limits_{\vx}~&s(\vx)+r(\vx) \mbox{ subject to } \vA^\top\vx=\vzero.
	\end{align*}
	The equivalence between PG-EXTRA and Condat-Vu can be obtained via considering the primal problem with an indicator function for the constraint~\cite{wu2018decentralized}. Here, we consider its dual problem in the following form:
	\begin{equation}\label{for:PG_EXTRA_dual_prob}
	\Min_\vy~{r^*}\square {s^*}(\vA\vy),
	\end{equation}
	where $r^*$ and $s^*$ are convex conjugate functions of $r$ and $s$, respectively. 
	We apply~\eqref{for:PD3O_nog} to~\eqref{for:PG_EXTRA_dual_prob} ($h\Rightarrow r^*,~l\Rightarrow s^*,~\vx\Rightarrow\vy,~\vs\Rightarrow\vt$) and arrive at 
	\begin{subequations}\label{for:PD3O_PG_EXTRA}
		\begin{align}
		\vz^{k+1}     & = (\vI-\tau\sigma \vA\vA^\top)\vt^k+{\sigma}\vA\vy^k-{\sigma}\nabla s(\vt^k),\label{for:PD3O_PG_EXTRA1}\\
		\vt^{k+1}     & =
		\argmin\limits_\vt~\{r(\vt)+{1\over{2\sigma}}\|\vt-\vz^{k+1}\|_F^2\},\label{for:PD3O_PG_EXTRA2}\\
		\vy^{k+1}     & =\vy^k-\tau\vA^\top \vt^{k+1}.\label{for:PD3O_PG_EXTRA3}
		\end{align}
	\end{subequations}
	Combining~\eqref{for:PD3O_PG_EXTRA1} and~\eqref{for:PD3O_PG_EXTRA3}, we get 
	\begin{align}\label{for:PD3O_PG_EXTRA_2step}
	\vz^{k+1} = \vz^k -\vt^k+(\vI-\tau\sigma \vA\vA^\top)(2\vt^k-\vt^{k-1})
	-{\sigma}\nabla s(\vt^k)+{\sigma}\nabla s(\vt^{k-1}).
	\end{align}
	We let $\tau\sigma={1\over2}$ and $\sigma=\alpha$, then~\eqref{for:PD3O_PG_EXTRA_2step} is exactly~\eqref{for:pg_extra_a} with $\vt\Rightarrow \vx$. 
	Because $\vM={2\tau^2}(\vI-(1/2)\vA\vA^\top)={\tau^2}(\vI+\vW)$ is positive definite, we can let $\vM_1=\vM$.
	If $\{\nabla s_i(x)\}_{i=1}^n$ are Lipschitz continuous with constant $L>0$, the other condition for convergence is 
	\begin{align*}
	\tau< 2\beta \leq {2\over L}\lambda_{\min}(\vM_1)= {2\tau^2\over L}\lambda_{\min}(\vI+\vW),
	\end{align*}
	where the second inequality comes from 
	\begin{align*}
	\langle \nabla s(\tilde\vx)-\nabla s(\bar\vx),\tilde\vx-\bar\vx\rangle \geq {1\over L}\|\tilde\vx-\bar\vx\|^2\geq {1\over L}\lambda_{\min}(\vM_1)\|\tilde\vx-\bar\vx\|_{\vM_1^{-1}}^2.
	\end{align*}
	Therefore, we obtain the condition on the stepsize 
	$$\alpha={1\over2\tau}<\lambda_{\min}(\vI+\vW)/L.$$
	This is exactly the upper bound in~\cite{shi2015proximal}.
	
	The previous upper bound is obtained with $\theta=1$. 
	As we mentioned before, we can choose $\theta$ to be close to $3/4$ to obtain large stepsizes.
	By letting $\theta=3/4+\epsilon$ with an arbitrary small $\epsilon>0$, we have $\vM_1={2\tau^2}(\vI-(3/4+\epsilon)(1/2)\vA\vA^\top)$ and $\vM_2=(1/4-\epsilon)\tau^2\vA\vA^\top$. 
	Then a larger upper bound for the stepsize 
	\begin{align*}
	\alpha={1\over2\tau}\leq&\lambda_{\min}(2\vI-(3/4+\epsilon)\vA\vA^\top)/L\\
	< &\lambda_{\min}(2\vI-(3/4)\vA\vA^\top)/L= ((3/4)\lambda_{\min}(\vI+\vW)+1/2)/L,
	\end{align*}
	is derived. 
	
	The new relaxed condition for $\vW$ is $\vM_1={\tau^2}(2\vI-(3/4+\epsilon)\vA\vA^\top)={\tau^2}((5/4-\epsilon)\vI+(3/4+\epsilon)\vW)$ being positive definite. That is $5\vI+3\vW$ is positive definite. Also, the special example in Subsection~\ref{sec:optimal_bdd} shows that the condition for the stepsize of PG-EXTRA can not be weakened. Its linear convergence without $\{r_i\}$ is discussed in~\cite{li2019linear} under the relaxed condition for $\vW$ and stepsize.

\section{Conclusion}
In this paper, we consider the primal-dual algorithm in~\cite{chen2013primal,drori2015simple,loris2011generalization} to solve the problem $f(\vx)+h\square l(\vx)$ and show its convergence under a weaker condition. We provide an example to show that this condition can not be weakened for a general problem. This result recovers and is more general than the positive-indefinite linear ALM proposed in~\cite{he2016positive}. Then we apply this result to decentralized consensus optimization and obtain the tight upper bound for the stepsize in PG-EXTRA.

\begin{acknowledgements}
	The authors would like to thank the anonymous reviewers for the helpful comments and suggestions that improve this paper.
\end{acknowledgements}

\bibliographystyle{spmpsci}      
\bibliography{PM3O_threeOverFour}

%
%


\end{document}